\documentclass[11pt]{amsart} 

\usepackage{amsmath, amsthm, amscd, amsfonts, amssymb, enumerate, verbatim, newlfont, calc, graphicx, color}
\usepackage[bookmarksnumbered, colorlinks, plainpages]{hyperref}
\usepackage{tikz}
 \usepackage{doi}

\usepackage{amsmath, amsthm, amscd, amsfonts, amssymb, enumerate,verbatim, newlfont, calc, graphicx, color}
\usepackage[bookmarksnumbered, colorlinks, plainpages]{hyperref}
\usepackage{tikz} 
\usepackage{doi}
\usepackage{cancel}
\newtheorem{theorem}{Theorem}[section] 
 
\newtheorem{fact}[theorem]{Fact}

\newtheorem{lemma}[theorem]{Lemma}
\newtheorem{proposition}[theorem]{Proposition}

\newtheorem{corollary}[theorem]{Corollary}
\newtheorem{definition}[theorem]{Definition}

\newtheorem{remark}[theorem]{Remark}

\newtheorem{example}[theorem]{Example}

\usepackage{mathtools}

\usepackage{pstricks}
\usepackage{epsfig}
\usepackage{pst-grad}
\usepackage{pst-plot}
\usepackage{subfig}

\begin{document}

\author[M. Nasernejad   and   J. Toledo]{Mehrdad  Nasernejad$^{1,2,*}$ and   Jonathan Toledo$^{3}$}
\title[Criteria for  the presence of the maximal  ideal]{Criteria for the presence of the maximal  ideal in the set of  associated primes}

\thanks{$^*$Corresponding author}

\thanks{E-mail addresses:  m$\_$nasernejad@yahoo.com  and  jonathan.toledo@infotec.mx}  
\maketitle

\begin{center}
{\it
$^{1}$Univ. Artois, UR 2462, Laboratoire de Math\'{e}matique de  Lens (LML), \\  F-62300 Lens, France \\ and \\
$^{2}$Universit\'e  Caen Normandie, ENSICAEN, CNRS, Normandie Univ, GREYC UMR  6072, F-14000 Caen,  France\\
$^{3}$INFOTEC Centro de investigaci\'{o}n e innovaci\'{o}n en informaci\'{o}n \\
 y comunicaci\'{o}n, Ciudad de M\'{e}xico,14050, M\'{e}xico
}
\end{center}

\vspace{0.4cm}

\begin{abstract}
In this paper, we establish some criteria to detect the presence of the maximal ideal $(x_1, \ldots, x_n)$ in the set of associated primes of powers of monomial ideals in the polynomial ring $K[x_1, \ldots, x_n]$. Furthermore, for each of these criteria, we illustrate its applicability with corresponding examples and applications.
 \end{abstract}
\vspace{0.4cm}


\section{Introduction and Overview}

Let $I$ be an ideal in a  commutative  Noetherian ring $R$. Then a prime ideal $\mathfrak{p}\subset  R$ is called an {\it associated prime} of $I$ if there exists an element $f\in R$  such that $\mathfrak{p}=(I:_R f)$, where $(I:_R f)=\{r\in R \mid  rf\in I\}$. The  {\it set of associated primes} of $I$, denoted by  $\mathrm{Ass}_R(R/I)$, is the set of all prime ideals associated to  $I$. In 1979,  Brodmann \cite{BR} proved  that the sequence $\{\mathrm{Ass}_R(R/I^s)\}_{s \geq 1}$ of associated prime ideals is stationary  for large $s$. That is to say, there exists a positive integer $s_0$ such that $\mathrm{Ass}_R(R/I^s)=\mathrm{Ass}_R(R/I^{s_0})$ for all $s\geq s_0$. The  minimal such $s_0$ is called the {\it index of stability}   of  $I$ and $\mathrm{Ass}_R(R/I^{s_0})$ is called the {\it stable set } 
 of associated prime ideals of  $I$,  denoted by $\mathrm{Ass}^{\infty }(I).$ Furthermore, assume that  $I=\bigcap_{i=1}^t\mathfrak{q}_i$ is 
  a \textit{minimal primary decomposition} of $I$ with $\sqrt{\mathfrak{q}_i}=\mathfrak{p}_i$ for all $i$,  that is, $\mathfrak{q}_i$ is a $\mathfrak{p}_i$-primary ideal, 
  $\sqrt{\mathfrak{q}_i}\neq \sqrt{\mathfrak{q}_j}$ for any $i\neq j$, and $I\neq \bigcap_{i\neq j}\mathfrak{q}_i$ for all $j =1, \ldots, t$. Then the  set of
  associated primes of  $I$ is equal to $\{\mathfrak{p}_1, \ldots, \mathfrak{p}_t\}$, consult  \cite[Theorem 4.5]{Atiyah}. 
 
 Now, let  $I$ be  a monomial  ideal  in a polynomial ring $R=K[x_1,\ldots,x_n]$ over a field $K$, $\mathfrak{m}=(x_1, \ldots, x_n)$ is  the  maximal ideal of $R$, and $x_1,\ldots,x_n$ are indeterminates.  
 One of the central questions in this area concerns whether the maximal ideal belongs to the set of associated primes. In fact, the available results in the literature are rather limited. Nevertheless, recent investigations have employed combinatorial methods to address this issue for (square-free) monomial ideals.  
For instance, Chen et al.\ proved in \cite[Lemma 3.1]{CMS} that if $G$ is an odd cycle of length $2k+1$ and $I$ denotes its edge ideal, then  
 $\mathrm{Ass}(R/I^n) = \mathrm{Min}(R/I)\cup \{\mathfrak{m}\},$ whenever $n \geq k+1.$   
Along the same lines, it is established in \cite[Proposition 3.6]{NKA} that for an odd cycle $G$ and its cover ideal $J$, one has $\mathfrak{m}\in \mathrm{Ass}(J^s)$ for every $s\geq 2$. 
Moreover, Herzog et al.\ in \cite[Corollaries 4.5 and 5.5]{HRV} examined the presence of the maximal ideal in the contexts of transversal polymatroidal ideals and ideals of Veronese type. In a similar spirit, \cite[Theorem 2.10]{N3} deals with monomial ideals generated by maximal-length paths in an unrooted starlike tree, and demonstrates conditions under which the maximal ideal arises. A related contribution by H\`a and Morey \cite[Corollary 3.6]{HM} provides a lower bound for the minimal power $m$ for which $I^m$ admits embedded primes.  

Additionally, the maximal ideal is also known to occur in certain families of monomial ideals referred to as nearly normally torsion-free monomial ideals. Examples include dominating ideals (see \cite[Theorem 4.4]{NBR} and \cite[Theorem 3.9]{NQBM}) as well as $t$-spread monomial ideals (cf.\ \cite[Lemma 5.15]{NQKR}). Furthermore, \cite[Lemma 3.5]{NQT} discusses the appearance of the maximal ideal in edge and cover ideals of cones of graphs.  

  It is also a well-established fact that  $\mathfrak{m}\in \mathrm{Ass}(R/I)$ iff $\mathrm{depth}(R/I)=0,$ refer to  Exercise 2.2.11 in \cite{V1}. For this reason, many works have approached the problem through the perspective of depth. Further details in this direction can be found in \cite{BHH, HNTT, HH2, MST, RS}.
 
  The major goal of Section \ref{Section3} is to develop some  criteria for identifying when the maximal ideal  $\mathfrak{m}=(x_1, \ldots, x_n)$ appears in the set of associated primes of powers of monomial ideals in  $R = K[x_1, \ldots, x_n].$  
The key contributions of this section  are presented in Theorem  \ref{Th.maximal-ideal.1}, Proposition  \ref{Pro.maximal-ideal.1}, Theorem \ref{Th.Not.maximal.1}, 
Lemma \ref{Lem.Maximal.Ass.1}, and Proposition \ref{Pro.Maximal.Ass.2}. In addition, we illustrate some  applications of these results, consult  Proposition \ref{App.1}, Examples \ref{App.2},  \ref{Exam.Not.maximal.1}, and  \ref{Exam.Maximal.Ass.2}.

 In Section \ref{Section4}, our aim is to explore the appearance of the maximal ideal in powers of nearly normally torsion-free monomial ideals, see Theorem \ref{NNTF}. In particular, in Example \ref{APP.2} we show how Theorem \ref{NNTF}  can be used.

Throughout this text, we let $\mathcal{G}(I)$ denote the unique minimal set of monomial generators of a monomial ideal $I \subset R = K[x_1, \ldots, x_n]$, where $R$ is the polynomial ring over a field $K$. Moreover, for a monomial $u \in R$, its {\em support}, written $\mathrm{supp}(u)$, is defined as the set of variables that divide $u$;  we also set $\mathrm{supp}(1) = \emptyset$. For a monomial ideal $I$, we further define  
$\mathrm{supp}(I) = \bigcup_{u \in \mathcal{G}(I)} \mathrm{supp}(u).$


\section{Preliminaries}

In this section, we collect fundamental definitions, results, and facts that will play a key role in the subsequent section of this paper. We start by
 recalling the following proposition.

\begin{proposition}\label{Pro.supp} (\cite[Proposition 4.2]{NKRT})
Let  $I$ be   a monomial ideal in $R=K[x_1, \ldots, x_n]$ over a field $K$ with $\mathcal{G}(I)=\{u_1, \ldots, u_m\}$ and $\mathrm{Ass}_R(R/I)=\{\mathfrak{p}_1, \ldots, \mathfrak{p}_s\}$. Then, the following statements hold. 
\begin{itemize}
\item[(i)] If $x_i|u_t$ for some $i$ with  $1\leq i \leq n$, and  
for some $t$ with  $1\leq t \leq m$, then there exists $j$ with  $1\leq j \leq s,$ 
such that $x_i\in \mathfrak{p}_j$. 
\item[(ii)] If $x_i\in \mathfrak{p}_j$ for some $i$ with  $1\leq i \leq n$, and  for some $j$ with   $1\leq j \leq s$, then  there exists $t$ with  
 $1\leq t \leq m$, such that  $x_i|u_t$. 
\end{itemize} 
Especially, $\bigcup_{j=1}^s \mathrm{supp}(\mathfrak{p}_j)=\bigcup_{t=1}^m \mathrm{supp}(u_t)$.
\end{proposition}


\begin{definition} (\cite[Definition 6.2.1]{MRS}) \label{Def. Corner elements}
\em{
Let $I$ be a monomial ideal in $R=K[x_1, \ldots, x_n]$. A monomial $f\in R$  is called an \textit{$I$-corner element} if $f\notin I$  and 
$x_1f, \ldots, x_nf \in I$.  The set of  $I$-corner elements of $I$  is denoted by $\mathrm{C}_R(I)$.
In particular,  $f$ is an $I$-corner element  of a monomial ideal $I \subset R=K[x_1, \ldots, x_n]$ if and only if $(I:f)=\mathfrak{m}\in \mathrm{Ass}(R/I)$, 
where $\mathfrak{m}=(x_1, \ldots, x_n)$.  
}
\end{definition}


\begin{fact}\label{fact2} (\cite[Exercise 2.1.62]{V1})
Let $R$  be a ring and $I$ an ideal. If $x \in R \setminus I$, 
 then there is an exact sequence of $R$-modules: 
 $$0\longrightarrow R/(I:x)\stackrel{\psi}{\longrightarrow} R/I \stackrel{\phi}{\longrightarrow} R/(I,x) \longrightarrow  0,$$
 where $\psi(\overline{r})= x\overline{r}$ is multiplication by $x$ and 
 $\phi(\overline{r}) =\overline{r}$.
\end{fact}


\begin{fact}\label{fact3} (\cite[Exercise 9.42]{sharp}) 
Let 
$$0\longrightarrow L{\longrightarrow} M {\longrightarrow} N  \longrightarrow  0,$$

be a short exact sequence of modules and homomorphisms over the commutative Noetherian ring $R$. Then, 
$\mathrm{Ass}(L) \subseteq  \mathrm{Ass}(M) \subseteq  \mathrm{Ass}(L) \cup  \mathrm{Ass}(N).$ 
\end{fact}


\begin{theorem} (\cite[Theorem 3.3]{SNQ})   \label{SNQ.Th.3.3}
 Let $I_1\subset R_1=K[x_1, \ldots, x_n]$ and  $I_2\subset R_2=K[y_1, \ldots, y_m]$ be two monomial ideals in disjoint sets of variables. Let 
$$I=I_1R+I_2R\subset R=K[x_1, \ldots, x_n, y_1, \ldots, y_m].$$ 
Then $\mathfrak{p}\in \mathrm{Ass}(R/I)$ if and only if $\mathfrak{p}=\mathfrak{p}_1R + \mathfrak{p}_2R$, where 
$ \mathfrak{p}_1\in \mathrm{Ass}(R_1/I_1)$ and $ \mathfrak{p}_2\in \mathrm{Ass}(R_2/I_2)$.
\end{theorem}



\begin{proposition}\label{PRO.WHEEL.1} (\cite[Proposition 3.10]{NKA})
Suppose that $W_{2n}$ is  a wheel graph of order $2n$ on the  vertex set $[2n]$, $R=K[x_1, \ldots, x_{2n}]$ is a  polynomial ring over a field $K$,
 and $\mathfrak{m}$ is the unique homogeneous maximal ideal  of $R$. Then  $\mathfrak{m}\in \mathrm{Ass}_R(R/(J(W_{2n}))^s)$  for all $s\geq 3$, and $\mathfrak{m}\notin \mathrm{Ass}_R(R/(J(W_{2n}))^s)$ for $s=1,2$.
 \end{proposition}


\begin{lemma}(\cite[Lemma 5.15]{NQKR}) \label{Lemma 5.15}
Let $u=x_ax_bx_n$ be a $t$-spread monomial and $I=B_t(u)\subset R=K[x_i:  x_i \in \mathrm{supp}(I)]$. Then we have the following:
\begin{enumerate}
\item[\em{(i)}] if $b < 2t+1$, then $I$ is normally torsion-free.
\item[\em{(ii)}] $a=1$ and $b\geq2t+1$, then $I$ is nearly normally torsion-free.
\item[\em{(iii)}] $a>1$ and $b\geq2t+1$, then $I$ is not nearly normally torsion-free.
\end{enumerate}
\end{lemma}


\begin{theorem}(\cite[Theorem 4.5]{NQKR}) \label{Theorem 4.5}
Assume that $G = (V(G), E(G))$ is a finite simple connected graph, and $J(G)$ denotes the cover ideal of $G$. 
Then $J(G)$ is nearly normally torsion-free if and only if $G$ is either a bipartite graph or an almost bipartite graph.
\end{theorem}


\begin{lemma}(\cite[Lemma 3.2]{NQKR}) \label{Lemma 3.2}
Let $I$ be a monomial ideal in a polynomial ring $R = K[x_1,\dots,x_n]$ over a field $K$ such that
 $\operatorname{Ass}_R(R/I)=\operatorname{Min}(I).$ Let $I(\mathfrak{m}\setminus \{x_i\})$ be normally torsion-free for all
$i=1,\dots,n$, where $\mathfrak{m}=(x_1,\dots,x_n).$ Then $I$ is nearly normally torsion-free.
\end{lemma}


\begin{theorem} (\cite[Theorem 2.5]{SN})  \label{NTF.Th.2.5}
Let $I$ be  a  monomial ideal  in  $R=K[x_1, \ldots, x_n]$ such that 
$I=I_1R + I_2R$, where
 $\mathcal{G}(I_1) \subset R_1=K[x_1, \ldots, x_m]$ and $\mathcal{G}(I_2) \subset R_2=K[x_{m+1}, \ldots, x_n]$
  for some  positive integer $m$. If $I_1$ and   $I_2$  are normally torsion-free, then  $I$  is so.
 \end{theorem}



\section{Criteria for the presence of the maximal  ideal in the set of  associated primes}\label{Section3}

In what follows, we provide some  criteria that determine whether the maximal ideal  $\mathfrak{m}=(x_1, \ldots, x_n)$  appears in the set of associated primes of powers of monomial ideals in $R=K[x_1, \ldots, x_n]$.  Our discussion begins with the following theorem.

\begin{theorem}    \label{Th.maximal-ideal.1}
Let $I\subset R=K[x_1, \ldots, x_n]$ be a monomial ideal, $\mathfrak{p}$ a monomial prime  ideal in $R$, $t\geq 1$, and   $y_1, \ldots, y_s$ be distinct variables in $R$ such that, for each $i=1, \ldots, s$, there exist  some  $\alpha_i \geq 1$ and a   monomial  ideal $J_i$ in $R$  with $y_i\notin \mathrm{supp}(J_i)$ such that
 $(I^t,y_i^{\alpha_i})=(J_i, y_i^{\alpha_i})$ and $\mathfrak{p}\setminus y_i \notin \mathrm{Ass}(R/J_i)$. 
 Then 
\begin{itemize}
\item[(i)] $\mathfrak{p}\in \mathrm{Ass}(R/I^t)$ if and only if $\mathfrak{p}\in \mathrm{Ass}(R/(I^t: \prod_{i=1}^sy_i^{\alpha_i}))$. 
\item[(ii)] If $\prod_{i=1}^{s}y_i^{\alpha_i} \in I^\ell$ for some $\ell \geq t$, then  $\mathfrak{p}\notin \mathrm{Ass}(R/I^t)$.
\end{itemize}
\end{theorem} 

\begin{proof}
(i) ($\Leftarrow$) Assume that  $\mathfrak{p}\in \mathrm{Ass}(R/(I^t: \prod_{i=1}^sy_i^{\alpha_i}))$. On account of  Fact \ref{fact2}, we have the 
 following short exact sequence 
$$0\longrightarrow R/(I^t: \prod_{i=1}^sy_i^{\alpha_i})\stackrel{\psi}{\longrightarrow} R/I^t \stackrel{\phi}{\longrightarrow} R/(I^t, \prod_{i=1}^sy_i^{\alpha_i}) \longrightarrow 0,$$
where $\psi(\overline{r})=\overline{r} \prod_{i=1}^sy_i^{\alpha_i}$ and $\phi(\overline{r})=\overline{r}$.  We can deduce from   Fact \ref{fact3} that 
$\mathrm{Ass}(R/(I^t: \prod_{i=1}^sy_i^{\alpha_i})) \subseteq \mathrm{Ass}(R/I^t)$. 
As  $\mathfrak{p}\in \mathrm{Ass}(R/(I^t: \prod_{i=1}^sy_i^{\alpha_i}))$, we get $\mathfrak{p}\in \mathrm{Ass}(R/I^t)$. 

($\Rightarrow$) Let $\mathfrak{p}\in \mathrm{Ass}(R/I^t)$. To establish  $\mathfrak{p}\in \mathrm{Ass}(R/(I^t: \prod_{i=1}^sy_i^{\alpha_i}))$, we use induction on $s$.  Since $y_1\notin \mathrm{supp}(J_1)$, we can derive from   Theorem \ref{SNQ.Th.3.3}  that  $\mathfrak{p}\in \mathrm{Ass}(R/(J_1, y_1^{\alpha_1}))$ if and only if $\mathfrak{p}=(\mathfrak{p}_1, y_1)$, where $\mathfrak{p}_1=\mathfrak{p}\setminus y_1 \in \mathrm{Ass}(R/J_1)$. 
Because  $(I^t,y_1^{\alpha_1})=(J_1,y_1^{\alpha_1})$ and   $\mathfrak{p}\setminus y_1 \notin \mathrm{Ass}(R/J_1)$, we obtain 
 $\mathfrak{p}\notin \mathrm{Ass}(R/(J_1, y_1^{\alpha_1}))$, and so $\mathfrak{p}\notin \mathrm{Ass}(R/(I^t, y_1^{\alpha_1}))$. 
By    Fact \ref{fact3}, we have 
$$\mathrm{Ass}(R/I^t) \subseteq \mathrm{Ass}(R/(I^t:y_1^{\alpha_1})) \cup \mathrm{Ass}(R/(I^t, y_1^{\alpha_1})).$$
This gives that  $\mathfrak{p}\in \mathrm{Ass}(R/(I^t:y_1^{\alpha_1}))$, and consequently, the claim is true for the case in which $s=1$. Now, assume that the assertion has been shown for  all values  less than $s>1$, and that $\mathfrak{p}\in \mathrm{Ass}(R/I^t)$.  To simplify the notation, put  $M:=\prod_{i=1}^{s-1}y_i^{\alpha_i}$.  From the induction hypothesis, we get  $\mathfrak{p}\in \mathrm{Ass}(R/(I^t:M))$. In addition,  according to  Fact \ref{fact3}, 
we have 
\begin{equation} \label{13}
\mathrm{Ass}(R/(I^t:M))\subseteq \mathrm{Ass}(R/((I^t:M):y_s^{\alpha_s}) )\cup \mathrm{Ass}(R/((I^t:M),y_s^{\alpha_s})).
\end{equation} 
Here, we  prove   $((I^t:M),y_s^{\alpha_s})=((J_s : M),y_s^{\alpha_s}).$  Due to  $y_s\notin \mathrm{supp}(J_s)$ and  
$(I^t,y_s^{\alpha_s})=(J_s, y_s^{\alpha_s})$,  one can easily  see that  $J_s  \subseteq I^t$, and therefore  
$((J_s :M),y_s^{\alpha_s}) \subseteq ((I^t:M),y_s^{\alpha_s}).$ To establish the reverse inclusion, take  a monomial $u$ in 
$((I^t:M),y_s^{\alpha_s})$. If $y_s^{\alpha_s}\mid u$, then $u\in ((J_s :M),y_s^{\alpha_s}),$ and the argument  is over.
 Let  $y_s^{\alpha_s}\nmid u$. Then we gain 
$u\in (I^t:M),$ and hence $uM\in I^t$. This implies  that there exists a monomial $f\in \mathcal{G}(I^t)$ such that $f\mid uM$. 
Since $f\in \mathcal{G}(I^t)$ and $(I^t,y_s^{\alpha_s})=(J_s, y_s^{\alpha_s})$, we must have $f\in J_s$ or $y_s^{\alpha_s} \mid f$. If 
$y_s^{\alpha_s} \mid f$, then $y_s^{\alpha_s} \mid uM$.  As $y_s \nmid M$, this gives that $y_s^{\alpha_s} \mid u$, which contradicts the fact that 
 $y_s^{\alpha_s}\nmid u$.  
We thus obtain  $f\in J_s$, and hence $u\in (J_s :M)$. Therefore, we can deduce that $((I^t:M),y_s^{\alpha_s}) \subseteq ((J_s :M),y_s^{\alpha_s}).$ 
Moreover, the assumption gives  that $\mathfrak{p}\setminus y_s \notin \mathrm{Ass}(R/J_s)$.  Because  $\mathrm{Ass}(R/(J_s:M)) \subseteq \mathrm{Ass}(R/J_s)$, we get  $\mathfrak{p}\setminus y_s \notin \mathrm{Ass}(R/(J_s:M))$. On account of  Theorem \ref{SNQ.Th.3.3}, we have 
$\mathfrak{p}\in \mathrm{Ass}(R/((J_s:M), y_s^{\alpha_s}))$ if and only if $\mathfrak{p}=(\mathfrak{p}_1, y_s)$, 
where $\mathfrak{p}_1 = \mathfrak{p}\setminus  y_s\in \mathrm{Ass}(R/(J_s:M))$. Consequently,   we must have 
$\mathfrak{p}\notin \mathrm{Ass}(R/((J_s:M), y_s^{\alpha_s}))$, and so $\mathfrak{p}\notin\mathrm{Ass}(R/((I^t:M), y_s^{\alpha_s})).$ 
It follows from $(\ref{13})$ that  $\mathfrak{p}\in \mathrm{Ass}(R/((I^t:M): y_s^{\alpha_s}))=\mathrm{Ass}(R/(I^t:\prod_{i=1}^sy_i^{\alpha_i}))$. 
This completes the inductive step, and thus the claim has been proved by induction. 

(ii) Suppose, on the contrary, that  $\mathfrak{p}\in \mathrm{Ass}(R/I^t)$.  It follows from  statement (i) that  
 $\mathfrak{p}\in \mathrm{Ass}(R/(I^t: \prod_{i=1}^{s}y_i^{\alpha_i})).$ Thanks to $\ell \geq t$ and $\prod_{i=1}^{s}y_i^{\alpha_i}\in I^\ell$, 
 we  obtain $(I^t: \prod_{i=1}^{s}y_i^{\alpha_i})=R$, which contradicts the fact that
  $\mathfrak{p}\in \mathrm{Ass}(R/(I^t: \prod_{i=1}^{s}y_i^{\alpha_i})).$ Accordingly, we conclude that   
  $\mathfrak{p}\notin \mathrm{Ass}(R/I^t)$.  
\end{proof}


As an immediate consequence of  Theorem  \ref{Th.maximal-ideal.1}, we obtain   the following corollary, which provides for us the first criterion in this paper.

\begin{corollary}    \label{Cor.maximal-ideal.1}
Let $I\subset R=K[x_1, \ldots, x_n]$ be a monomial ideal, $\mathfrak{m}=(x_1, \ldots, x_n)$, and   $y_1, \ldots, y_s$ be distinct variables in $R$ such that, for each $i=1, \ldots, s$, there exist  some  $\alpha_i \geq 1$ and a   monomial  ideal $J_i$ in $R$  with $y_i\notin \mathrm{supp}(J_i)$ such that
 $(I,y_i^{\alpha_i})=(J_i, y_i^{\alpha_i})$ and $\mathfrak{m}\setminus y_i \notin \mathrm{Ass}(R/J_i)$. If  $\prod_{i=1}^{s}y_i^{\alpha_i} \in I$, 
 then $\mathfrak{m}\notin  \mathrm{Ass}(R/I)$. 
\end{corollary}


The following proposition  highlights an application of   Corollary \ref{Cor.maximal-ideal.1}, which also provides an alternative proof of Corollary 2.9 in  \cite{NR}. 

\begin{proposition} \label{App.1}
Let $I\subset R=K[x_1, \ldots, x_n]$ be a square-free monomial ideal. Then $\mathfrak{m} \in \mathrm{Ass}(R/I)$ if and only if $I=\mathfrak{m}$.
\end{proposition}

\begin{proof}
If $I=\mathfrak{m}$, then  $\mathrm{Ass}(R/I)=\{\mathfrak{m}\}$, and so $\mathfrak{m} \in \mathrm{Ass}(R/I)$. 
 Conversely, let $\mathfrak{m} \in \mathrm{Ass}(R/I)$.  We proceed by induction on $n$. If $n=1$, then we have $R=K[x_1]$, $\mathfrak{m}=(x_1)$, and 
  $I=(x_1)$. Since $\mathrm{Ass}(R/I)=\{(x_1)\}$, we deduce that the claim holds for $n=1$. 
  Now, assume that $n>1$, $\mathfrak{m} \in \mathrm{Ass}(R/I)$, 
   and that the claim has been shown   for all values less than $n$. Because  $\mathfrak{m} \in \mathrm{Ass}(R/I)$, according to Proposition \ref{Pro.supp}, we must have    $\mathrm{supp}(I)=\{x_1, \ldots, x_n\}$. Suppose, on the contrary, that $I\neq \mathfrak{m}$. This means that there exists some $u\in \mathcal{G}(I)$ 
   such that $|\mathrm{supp}(u)|\geq 2$. Without loss of generality, assume that $u=\prod_{i=1}^m x_i$ with $2\leq m \leq n$. In particular, it follows from $u\in \mathcal{G}(I)$ that    $x_i\notin I$ for all $i=1, \ldots, m$.  Since $I$ is square-free, for all $i=1, \ldots, m$, we can write $I=x_iH_i+ J_i$
    with $x_i\notin \mathrm{supp}(J_i)$.  In particular, we have $(I, x_i)=(J_i, x_i)$ for all $i=1, \ldots, m$. 
    If $\mathfrak{m}\setminus x_i \notin \mathrm{Ass}(R/J_i)$  for all $i=1, \ldots, m$, and by remembering that $\prod_{i=1}^m x_i \in I$, then Corollary 
    \ref{Cor.maximal-ideal.1} yields that $\mathfrak{m}\notin  \mathrm{Ass}(R/I)$, which contradicts our assumption. Consequently, we obtain there exists some 
    $1\leq j \leq m$ such that  $\mathfrak{m}\setminus x_j \in \mathrm{Ass}(R/J_j)$. It follows now from the  induction hypothesis  that $J_j=\mathfrak{m}\setminus x_j$. 
    Thanks to $(I, x_j)=(J_j, x_j)$, we get $(I,x_j)=\mathfrak{m}$. This implies that $x_s\in I$ for all $s\in \{1, \ldots, m\}\setminus \{j\}$, while $x_i\notin I$ for all $i=1, \ldots, m$. 
    This leads to a contradiction. In addition, on account of   $\mathrm{supp}(I)=\{x_1, \ldots, x_n\}$, we can deduce that $I=\mathfrak{m}$. 
    This completes the inductive step, and therefore the assertion  has been shown  by induction.
\end{proof}


Corollary   \ref{Cor.maximal-ideal.2} follows directly from Proposition  \ref{Pro.maximal-ideal.1}, which sets forth the second   criterion. To reach this purpose, we 
state and prove Proposition  \ref{Pro.maximal-ideal.1} as follows.

\begin{proposition}    \label{Pro.maximal-ideal.1}
Let $I\subset R=K[x_1, \ldots, x_n]$ be a monomial ideal,  $\mathfrak{p}$ a monomial prime  ideal in $R$,  $t\geq 1$,  $\alpha_1, \ldots, \alpha_s$ be positive integers, and  $y_1, \ldots, y_s$ be distinct variables in $R$  such that  $\mathfrak{p}\notin \mathrm{Ass}(R/(I^t,  y_{1}^{\alpha_1}))$ and, for each  $i=2, \ldots, s$,
  $\mathfrak{p}\notin \mathrm{Ass}(R/((I^t:\prod_{j=1}^{i-1} y_j^{\alpha_j}), y_{i}^{\alpha_{i}}))$.  Then 
\begin{itemize}
\item[(i)] $\mathfrak{p}\in \mathrm{Ass}(R/I^t)$ if and only if $\mathfrak{p}\in \mathrm{Ass}(R/(I^t: \prod_{i=1}^sy_i^{\alpha_i}))$. 
\item[(ii)] If $\prod_{i=1}^{s}y_i^{\alpha_i} \in I^\ell$ for some $\ell \geq t$, then  $\mathfrak{p}\notin \mathrm{Ass}(R/I^t)$.
\end{itemize}
\end{proposition} 

\begin{proof}
(i) ($\Leftarrow$) It can be shown by a similar proof in Theorem  \ref{Th.maximal-ideal.1}(i). 

($\Rightarrow$) Let $\mathfrak{p}\in \mathrm{Ass}(R/I^t)$. To show   $\mathfrak{p}\in \mathrm{Ass}(R/(I^t: \prod_{i=1}^sy_i^{\alpha_i}))$,  we proceed by induction on  $s$. Based on   Fact \ref{fact3}, we have 
$$\mathrm{Ass}(R/I^t) \subseteq \mathrm{Ass}(R/(I^t:y_1^{\alpha_1})) \cup \mathrm{Ass}(R/(I^t, y_1^{\alpha_1})).$$
Due to  $\mathfrak{p}\notin \mathrm{Ass}(R/(I^t, y_1^{\alpha_1}))$, this implies  that  $\mathfrak{p}\in \mathrm{Ass}(R/(I^t:y_1^{\alpha_1}))$, and so the assertion  is true for the case in which $s=1$. Suppose that our claim has been proved for all values  less than $s>1$, and that $\mathfrak{p}\in \mathrm{Ass}(R/I^t)$. Set   $M:=\prod_{j=1}^{s-1}y_j^{\alpha_j}$. It follows from the induction hypothesis that   $\mathfrak{p}\in \mathrm{Ass}(R/(I^t:M))$. 
According to  Fact \ref{fact3}, we have 
\begin{equation} \label{14}
\mathrm{Ass}(R/(I^t:M))\subseteq \mathrm{Ass}(R/((I^t:M):y_s^{\alpha_s}) )\cup \mathrm{Ass}(R/((I^t:M),y_s^{\alpha_s})).
\end{equation} 
In  light of  $\mathfrak{p}\notin\mathrm{Ass}(R/((I^t:M), y_s^{\alpha_s}))$, one can conclude from 
 $(\ref{14})$ that  $\mathfrak{p}\in \mathrm{Ass}(R/((I^t:M): y_s^{\alpha_s}))=\mathrm{Ass}(R/(I^t:\prod_{i=1}^sy_i^{\alpha_i}))$. 
This completes the inductive step, and therefore the assertion  has been shown  by induction. 

(ii) It is sufficient    to mimic the proof of Theorem  \ref{Th.maximal-ideal.1}(ii).
\end{proof}


\begin{corollary}    \label{Cor.maximal-ideal.2}
Let $I\subset R=K[x_1, \ldots, x_n]$ be a monomial ideal, $\mathfrak{m}=(x_1, \ldots, x_n)$,  $\alpha_1, \ldots, \alpha_s$ be positive integers, and  $y_1, \ldots, y_s$ be distinct variables in $R$  such that   $\mathfrak{m}\notin \mathrm{Ass}(R/(I,  y_{1}^{\alpha_1}))$   and,   for each  $i=2, \ldots, s$, we have   
$\mathfrak{m}\notin \mathrm{Ass}(R/((I:\prod_{j=1}^{i-1} y_j^{\alpha_j}), y_{i}^{\alpha_{i}}))$.  
If  $\prod_{i=1}^{s}y_i^{\alpha_i} \in I,$  then $\mathfrak{m}\notin  \mathrm{Ass}(R/I)$. 
\end{corollary}


To demonstrate the utility of  Corollary \ref{Cor.maximal-ideal.2},  we present   the following  example. We first review  its background. 
Let $t\geq 2$ be an integer, and  $I=(x^t, xy^{t-2}z, y^{t-1}z)$ be a monomial ideal in $R=K[x,y,z]$ and $\mathfrak{m}=(x,y,z)$.   It has already been shown in
 \cite[Proposition 2.5]{MST} that 
  \[
\mathrm{depth}(R/I^s) =
    \begin{dcases}
    1 & \text{if }  s \leq t-1 \\
   0   & \text{if } s \geq t. \\
                \end{dcases}
\]
  In addition, note that  $I$ has the following minimal primary  decomposition 
 $$I=(x,y^{t-1}) \cap (x^t, y^{t-2}) \cap (x^t,z).$$  This implies that $\mathrm{Min}(I)= \{(x,y), (x,z)\}.$ Moreover,  since  $\mathfrak{m}\in \mathrm{Ass}(R/I^s)$  if and  only if $\mathrm{depth}(R/I^s)=0$,  we can rewrite the above result as follows 
  \[
\mathrm{Ass}(R/I^s) =
    \begin{dcases}
    \{(x,y), (x,z)\} & \text{if } s \leq t-1 \\
\{(x,y), (x,z), \mathfrak{m}\}  & \text{if } s \geq t. \\
                \end{dcases}
\]
For $s\geq t$, in  \cite[Proposition 2.5]{MST}, the authors provided an $I^s$-corner element, that is, $u:=x^{ts-t^2+t}y^{t^2-2t}z^{t-1}$, such that  
$\mathfrak{m}=(I^s:u)$, and so $\mathfrak{m}\in \mathrm{Ass}(R/I^s)$. This yields that $\mathrm{depth}(R/I^s)=0$ for all  $s\geq t$. 
To show $\mathrm{depth}(R/I^s)=1$  for all $s\leq t-1$, the authors relied on the theory of Buchberger graphs.  
In the following example, we prove that $\mathfrak{m} \notin \mathrm{Ass}(R/I^s)$ for $t = 5$ and $s = 3$, by applying Corollary~\ref{Cor.maximal-ideal.2}.
\begin{example}\label{App.2}
\em{
Consider the monomial ideal  $I=(x^5, xy^3z, y^4z)$  in the polynomial ring $R=K[x,y,z]$.
 To use Corollary~\ref{Cor.maximal-ideal.2}, it is enough to check the following claims:
 
 \bigskip
\textbf{Claim 1:} $\mathfrak{m}\notin \mathrm{Ass}(R/(I^3, z))$. Using the binomial expansion, we obtain
\begin{align}
I^3
&= \sum_{a+b+c=3} (x^5)^a (x y^{3} z)^b (y^{4} z)^cR \label{4.1} \\
&= x^{15}R + \sum_{\substack{a+b+c=3,~ 1\leq b+c\leq 3}}
(x^5)^a (x y^{3} z)^b (y^{4} z)^cR.
\notag
\end{align}
On account of $b+c\geq 1$, we must have $b\geq 1$ or $c\geq 1$, and so we can deduce 
\[
\sum_{\substack{a+b+c=3,~ 1\leq b+c\leq 3}}
(x^5)^a (x y^{3} z)^b (y^{4} z)^cR \subseteq zR.
\]
This implies  that $(I^3, z)=(x^{15}, z)$, and hence $\mathrm{Ass}(R/(I^3, z)) = \{(x,z)\}.$ 
Therefore, $\mathfrak{m}\notin \mathrm{Ass}(R/(I^3, z))$, as claimed.

\bigskip
\textbf{Claim 2:} $\mathfrak{m}\notin \mathrm{Ass}(R/((I^3 : z), y^{3}))$.
 From (\ref{4.1}), we get the following 
\[
(I^3 : z)
= x^{15}R
+ \sum_{\substack{a+b+c=3,~ 1\leq b+c\leq 3}}
x^{5a+b}\, y^{3b+4c}\, z^{b+c-1}R.
\]
It follows from  $b+c\geq 1$ that  $b\geq 1$ or $c\geq 1$. If $b\geq 1$ (resp. $c\geq 1$), then
$b(t-2)=3b\geq 3$ (resp. $c(t-1)=4c\geq 3$). Hence, we can derive that
\[
\sum_{\substack{a+b+c=3,~ 1\leq b+c\leq 3}}
x^{5a+b}\, y^{3b+4c}\, z^{b+c-1}R
\subseteq y^{3}R.
\]
Thus, $((I^3 : z), y^{3}) = (x^{15}, y^{3}),$ and so $\mathrm{Ass}(R/((I^3 : z), y^{3})) = \{(x,y)\}.$ 
This gives rise to $\mathfrak{m}\notin \mathrm{Ass}(R/((I^3 : z), y^{3})),$ as required.
 
\bigskip
\textbf{Claim 3:}  $\mathfrak{m}\notin \mathrm{Ass}\big(R/((I^3 : zy^{3}), x^{11})\big).$ Note that $x^{11}\in (I^3 : zy^{3})$. 
 In view of  (\ref{4.1}), we can deduce  the following decomposition
 \begin{align*}
(I^3 : zy^{3}) &=(y^9 z^2,x y^8 z^2,x^2 y^7 z^2,x^3 y^6 z^2, x^5 y^5 z,x^6 y^4 z, x^7 y^3 z,x^{10} y,x^{11}:zy^3)\\
&=(y^{9}, xy^{8}, x^{2}y^{7}, x^{3}y^{6}, x^{5}y^{5}, x^{6}y^{4}, x^{7}y^{3}, x^{10}y, x^{11}) \cap
 (z^{2}, x^{5}z, x^{10}).
 \end{align*}
This implies that  $\mathfrak{m}\notin \mathrm{Ass}\big(R/((I^3 : zy^{3}), x^{11})\big).$
 
Since $x^{11}y^3z=(x^5)^2(xy^3z)\in I^3$, it follows now from Corollary \ref{Cor.maximal-ideal.2} that $\mathfrak{m} \notin \mathrm{Ass}(R/I^3)$, as desired. 
}
\end{example}


Here, we provide the next criterion to determine the presence of the maximal ideal in the following theorem. 

\begin{theorem} \label{Th.Not.maximal.1}
Suppose that    $I \subset R=K[x_1, \ldots, x_n]$ with $n\geq 2$,   is   a monomial ideal, $\mathfrak{m}=(x_1, \ldots, x_n)$, and 
  $\mathcal{G}(I)=\{x^{r_{1,1}}_1\cdots x^{r_{1,n}}_n, \ldots, x^{r_{k,1}}_1\cdots x^{r_{k,n}}_n\}$
  such that  there exist some  $1\leq i\neq j \leq n$ with 
  $$r_{1,i} \geq r_{2,i} \geq \cdots \geq r_{k, i} \quad \text{and} \quad r_{1,j} \geq r_{2,j} \geq \cdots \geq r_{k, j}.$$ 
   Then   $\mathfrak{m}\notin \mathrm{Ass}(R/I)$. 
\end{theorem}

\begin{proof} 
Without loss of generality, we may assume that   $r_{1,i} \geq r_{2,i} \geq \cdots \geq r_{k, i}$ for $i=1,2$. 
On the contrary, assume that $\mathfrak{m}\in \mathrm{Ass}(R/I)$. This implies that there exists some monomial $h\in R$ such that  $h\notin I$ and 
  $x_ih\in I$ for all $i=1, \ldots, n$. In particular, we have $x_1h, x_2h \in I$.   Let $h=\prod_{i=1}^n x_i^{\theta_i}$. We thus get 
 $x_1^{\theta_1 +1} \prod_{i=2}^n x_i^{\theta_i}\in I$ and $x_1^{\theta_1} x_2^{\theta_2+1}\prod_{i=3}^n x_i^{\theta_i} \in I$. 
  This gives  that there exist  $1\leq \lambda \leq k$  and $1\leq \gamma \leq k$  such that 
  \begin{equation} \label{15}
  x^{r_{\lambda,1}}_1\cdots x^{r_{\lambda,n}}_n \mid x_1^{\theta_1 +1} \prod_{i=2}^n x_i^{\theta_i} \;  \text {and} \;  
  x^{r_{\gamma,1}}_1\cdots x^{r_{\gamma,n}}_n \mid x_1^{\theta_1} x_2^{\theta_2+1}\prod_{i=3}^n x_i^{\theta_i}.
  \end{equation}
       By virtue  of    $\prod_{i=1}^n x_i^{\theta_i}\notin I$,   we must have
      \begin{equation} \label{16}
      x^{r_{\lambda,1}}_1\cdots x^{r_{\lambda,n}}_n \nmid  \prod_{i=1}^n x_i^{\theta_i} \;\; \text{and} \;\;  
     x^{r_{\gamma,1}}_1\cdots x^{r_{\gamma,n}}_n \nmid  \prod_{i=1}^n x_i^{\theta_i}.
      \end{equation}
     It follows from (\ref{15}) and (\ref{16})  that  $\theta_1 < r_{\lambda,1} \leq \theta_1+1$, $\theta_2 < r_{\gamma,2} \leq \theta_2+1$, 
   $r_{\lambda, 2} \leq \theta_2$, and $r_{\gamma, 1} \leq \theta_1$. Consequently, we  obtain $r_{\gamma, 1} \leq \theta_1 < r_{\lambda,1}$ and 
   $r_{\lambda, 2} \leq \theta_2< r_{\gamma,2}$. One can immediately deduce  from  our assumption  that 
    $\gamma < \lambda$ and $\lambda < \gamma$. This leads to a contradiction. Accordingly, we conclude that   $\mathfrak{m}\notin \mathrm{Ass}(R/I)$.    
\end{proof}


\begin{example} \label{Exam.Not.maximal.1}
\em{
Let $R = K[x,y,z,t]$ and consider the monomial ideal
\[
I = (x^5y z^4 ,\; x^4 z^3 t^2,\; x^3  y^2 z^2,\; x^2 z t^3).
\]
Then $\mathcal{G}(I) = \{x^5y z^4 ,\; x^4 z^3 t^2,\; x^3  y^2 z^2,\; x^2 z t^3\}.$ 
Note that the exponents of $x$ satisfy $5 \geq 4 \geq 3 \geq 2$, and the exponents of $z$ satisfy $4 \geq 3 \geq 2 \geq 1$. 
Hence the hypothesis of Theorem~\ref{Th.Not.maximal.1} is fulfilled. Therefore, by the theorem we obtain
 $\mathfrak{m} = (x,y,z,t) \notin \mathrm{Ass}(R/I).$
 }
\end{example}


For the rest of this section, we shall proceed according to the following criteria. To this end, we establish the following lemma, whose proof 
will be applied in Proposition  \ref{Pro.Maximal.Ass.2}.

\begin{lemma}\label{Lem.Maximal.Ass.1}
Let $L=uI+J \subset R=K[x_1, \ldots, x_n]$ be a monomial ideal such that $\mathrm{supp}(u) \cap (\mathrm{supp}(I) \cup \mathrm{supp}(J))=\emptyset$. 
Then, for all $t\geq 1$,  if $\mathfrak{m}\in \mathrm{Ass}(L^t)$, then  $\mathfrak{m}\in \mathrm{Ass}(L^t,u^t)$. 
\end{lemma}

\begin{proof}
 Fix $t\geq 1$ and let $\mathfrak{m}\in \mathrm{Ass}(L^t)$.  Note that $(L:u)=I+J$. We first  demonstrate  that  $(L^t:u^t)=(L:u)^t$.  By 
the binomial expansion theorem, we have  $(I+J)^t=\sum_{r=0}^tI^rJ^{t-r}$, and therefore  we get the following equalities 
  $$(L^t:u^t)=\sum_{r=0}^t (u^r I^r J^{t-r}:u^t)=\sum_{r=0}^t  I^r J^{t-r}=(L:u)^t.$$   
 We now verify that   $L^t=(L^t:u^t) \cap (L^t,u^t)$. By  virtue of  $L^t \subseteq (L^t:u^t)$ and $L^t \subseteq  (L^t,u^t)$,  we obtain    $L^t \subseteq (L^t:u^t) \cap (L^t,u^t)$.  To show  the reverse inclusion, pick a  monomial  $f\in (L^t:u^t) \cap (L^t,u^t)$. This implies that $fu^t \in L^t$ and $f\in (L^t,u^t)$.    It follows from $f\in (L^t, u^t)$ that $f\in L^t$ or $u^t \mid f$.  On the contrary, assume that  $f\notin L^t$.  Then 
  $u^t  \mid  f$.    Based on the binomial expansion theorem,  we get   $L^t=\sum_{r=0}^t u^r I^r J^{t-r}$.  Hence, we get $fu^t \in  u^r I^r J^{t-r}$ for some 
   $0\leq r \leq t$. This yields that  $f \in  (u^r I^r J^{t-r}: u^t)$. As   $0\leq r \leq t$ and $\mathrm{supp}(u) \cap (\mathrm{supp}(I) \cup \mathrm{supp}(J))=\emptyset$, we obtain   $f \in   I^r J^{t-r}$. Due to $r\leq t$ and $u^t \mid f$, this gives that $u^r\mid f$, and so $f\in u^rR$. Consequently, we have 
   $f\in u^rR \cap I^r J^{t-r}$. Because $\mathrm{supp}(u) \cap (\mathrm{supp}(I) \cup \mathrm{supp}(J))=\emptyset$, this leads to 
  $ u^rR \cap I^r J^{t-r}= u^r I^r J^{t-r}$, and so  $f\in u^r I^r J^{t-r}$. This implies that  $f\in L^t$, which is a contradiction. Accordingly, we must have 
$f\in L^t$, and hence $L^t=(L^t:u^t) \cap (L^t,u^t).$ Since $(L^t:u^t)=(L:u)^t$, we deduce that  $L^t=(L:u)^t \cap (L^t,u^t)$ is  a  decomposition of   $L^t$. Due to $(L:u)^t=(I+J)^t$ and $\mathrm{supp}(u) \cap (\mathrm{supp}(I) \cup \mathrm{supp}(J))=\emptyset$, Proposition \ref{Pro.supp} yields that 
  $\mathfrak{m}\notin \mathrm{Ass}(L:u)^t$, and by virtue of $\mathfrak{m}\in \mathrm{Ass}(L^t)$, we can derive that  $\mathfrak{m}\in  \mathrm{Ass}(L^t,u^t)$.   This completes the proof. 
\end{proof}


\begin{remark}
\em{
It should be noted that the converse of Lemma \ref{Lem.Maximal.Ass.1} may not hold. To see a  counterexample, assume that  
 $L:=(x^{11}z, x^5y^4, x^6y^2, y^{11}z)$ in $R=K[x,y,z]$.  Put $I:=(x^{11}, y^{11})$ and $J:=(x^5y^4, x^6y^2)$. Then we can write $L=zI+J$
 such that $z\notin \mathrm{supp}(I) \cup \mathrm{supp}(J)$.  Using {\it Macaulay2} \cite{GS}, we obtain $(x,y,z)\in \mathrm{Ass}(L^3, z^3)$,  while 
 $(x,y,z)\notin \mathrm{Ass}(L^3)$. 
}
\end{remark}


\begin{proposition}\label{Pro.Maximal.Ass.2}
Let $L=uI+J \subset R=K[x_1, \ldots, x_n]$ be a monomial ideal such that  $\mathrm{supp}(u) \cap (\mathrm{supp}(I) \cup \mathrm{supp}(J))=\emptyset$. Then, for all $t\geq 1$,  if $\mathfrak{m}\in \mathrm{Ass}(L^t,u^t)$,  then at least one of the following statements holds:
\begin{itemize}
\item[(i)] $\mathfrak{m}\in \mathrm{Ass}(L^t)$;
\item[(ii)] $u=x_j$ for some $1\leq j \leq n$ and $\mathfrak{m}\setminus \{x_j\}\in \mathrm{Ass}(R/(I+J)^t)$. 
\end{itemize}
\end{proposition}

\begin{proof}
For simplicity of notation, put $A:=(L^t: u^t)$ and $B:=(L^t,u^t)$. We consider the following short exact sequence
$$0\longrightarrow \frac{R}{A\cap B} {\longrightarrow} \frac{R}{A} \oplus \frac{R}{B} {\longrightarrow} \frac{R}{A+B}  \longrightarrow 0.$$
 By virtue of Fact \ref{fact3}, we can conclude  the following inclusion
 $$\mathrm{Ass}(\frac{R}{A})\cup \mathrm{Ass}(\frac{R}{B})=\mathrm{Ass}(\frac{R}{A} \oplus \frac{R}{B}) \subseteq \mathrm{Ass}(\frac{R}{A\cap B}) \cup \mathrm{Ass}(\frac{R}{A+B}).$$ 
 Based on the proof of Lemma \ref{Lem.Maximal.Ass.1}, we have $A=(L:u)^t=(I+J)^t$. Moreover, the assumption says that $\mathfrak{m}\in \mathrm{Ass}(R/B)$. 
 Hence, we obtain  $$\mathfrak{m}\in \mathrm{Ass}(\frac{R}{A\cap B}) \cup \mathrm{Ass}(\frac{R}{A+B}).$$ 
 Also, according to the proof of Lemma \ref{Lem.Maximal.Ass.1}, we have $L^t=A\cap B$. If $\mathfrak{m}\in \mathrm{Ass}(L^t)$, then we obtain  statement (i). 
 Hence,  let  $\mathfrak{m}\notin \mathrm{Ass}(L^t)$. This implies that $\mathfrak{m}\in \mathrm{Ass}(R/(A+B))$. Furthermore, we have the following equalities $$A+B=(I+J)^t+\sum_{i=0}^tu^iI^iJ^{t-i} + u^tR=(I+J)^t+ u^tR.$$ 
 Thanks to  $\mathrm{supp}(u) \cap (\mathrm{supp}(I) \cup \mathrm{supp}(J))=\emptyset$, we can derive from   Theorem \ref{SNQ.Th.3.3} that 
 $m\setminus \{x_i\} \in \mathrm{Ass}(R/(I+J)^t)$ for some $x_i\in \mathrm{supp}(u)$. It follows now from Proposition \ref{Pro.supp} 
 that $|\mathrm{supp}(I+J)|=n-1$. 
 Therefore, we must have  $|\mathrm{supp}(u)|=1$, say  $u=x_j$, where $1\leq j \leq n$. This gives rise to  $\mathfrak{m}\setminus \{x_j\}\in \mathrm{Ass}(R/(I+J)^t)$. 
 We thus get  statement (ii), as claimed.
 \end{proof}


The following example illustrates that $\mathfrak{m}$  may belong to $\mathrm{Ass}(L^t,u^t)$, even though both statements of Proposition \ref{Pro.Maximal.Ass.2}
 hold.
 
\begin{example}\label{Exam.Maximal.Ass.2}
\em{
Consider the wheel graph $W_6$, shown in Figure 1.  Also, let $L$ denote its cover ideal in the polynomial ring $R=K[x_1, \ldots, x_6]$, that is, 
\begin{align*}
L=&(x_1,x_2) \cap (x_2,x_3)\cap (x_3, x_4) \cap (x_4,x_5)\cap (x_5,x_1)\cap (x_1,x_6)\cap (x_2, x_6)\\
&\cap (x_3,x_6) \cap (x_4,x_6) \cap(x_5,x_6)\\
=&(x_2x_4x_5x_6,x_2x_3x_5x_6,x_1x_3x_5x_6,x_1x_3x_4x_6,x_1x_2x_4x_6,x_1x_2x_3x_4x_5).
\end{align*}
\begin{figure}[h]
\centering
\scalebox{1.2}{
\begin{pspicture}(0,-1.668125)(3.6828125,1.668125)
\psdots[dotsize=0.12](1.8209375,1.2296875)
\psdots[dotsize=0.12](2.8209374,0.2496875)
\psdots[dotsize=0.12](0.8209375,0.2496875)
\psdots[dotsize=0.12](1.1809375,-0.7503125)
\psdots[dotsize=0.12](2.3809376,-0.7703125)
\psline[linewidth=0.04cm](0.8209375,0.2496875)(1.8409375,1.2696875)
\psline[linewidth=0.04cm](1.8209375,1.2496876)(2.8009374,0.2696875)
\psline[linewidth=0.04cm](0.8009375,0.2296875)(0.8409375,0.3096875)
\psline[linewidth=0.04cm](0.8209375,0.2496875)(1.2009375,-0.7903125)
\psline[linewidth=0.04cm](1.1809375,-0.7503125)(2.3409376,-0.7503125)
\psdots[dotsize=0.12](1.8009375,0.0896875)
\psline[linewidth=0.04cm](2.3809376,-0.7703125)(2.8209374,0.2496875)
\psline[linewidth=0.04cm](1.8009375,1.2096875)(1.8609375,1.2096875)
\psline[linewidth=0.04cm](0.8209375,0.2496875)(1.8209375,0.1096875)
\psline[linewidth=0.04cm](1.1609375,-0.7103125)(1.1809375,-0.7103125)
\psline[linewidth=0.04cm](1.2009375,-0.7303125)(1.8009375,0.0896875)
\psline[linewidth=0.04cm](1.8209375,0.1096875)(2.3809376,-0.7703125)
\psline[linewidth=0.04cm](1.8209375,0.1096875)(2.8209374,0.2496875)
\psline[linewidth=0.04cm](1.7809376,0.1096875)(1.8009375,0.0896875)
\psline[linewidth=0.04cm](1.8009375,0.1096875)(1.8209375,1.2296875)
\usefont{T1}{ptm}{m}{n}
\rput(1.7823437,1.4796875){$x_1$}
\usefont{T1}{ptm}{m}{n}
\rput(3.2023437,0.2596875){$x_2$}
\usefont{T1}{ptm}{m}{n}
\rput(2.8614063,-0.7203125){$x_3$  }
\usefont{T1}{ptm}{m}{n}
\rput(0.76234376,-0.7203125){$x_4$}
\usefont{T1}{ptm}{m}{n}
\rput(0.40234375,0.2796875){$x_5$}
\usefont{T1}{ptm}{m}{n}
\rput(2.1023438,0.4196875){$x_6$}
\usefont{T1}{ptm}{m}{n}
\end{pspicture}}
\caption{The wheel graph $W_6$.}
\label{fig:wheel}
\end{figure}

Put $I:=(x_2x_4x_5,x_2x_3x_5,x_1x_3x_5,x_1x_3x_4,x_1x_2x_4)$ and $J:=(x_1x_2x_3x_4x_5)$. Then we can write $L=x_6I+J$. 
Note that $x_6\notin \mathrm{supp}(I) \cup \mathrm{supp}(J)$. Using {\it Macaulay2} \cite{GS}, we get 
$\mathfrak{m}=(x_1,x_2,x_3,x_4,x_5,x_6)\in \mathrm{Ass}(L^3, x_6^3)$ and 
$\mathfrak{m}\setminus\{x_6\}=(x_1,x_2,x_3,x_4,x_5)\in \mathrm{Ass}(R/(I+J)^3)$ . 
On the other hand, it follows from Proposition \ref{PRO.WHEEL.1} that $\mathfrak{m}=(x_1,x_2,x_3,x_4,x_5,x_6)\in \mathrm{Ass}(L^3)$. 
This shows that both statements of Proposition \ref{Pro.Maximal.Ass.2}   are valid.
}
\end{example}


\section{The appearance of the maximal ideal  in powers of nearly normally torsion-free monomial ideals}\label{Section4}

 In this section, our aim is to explore the appearance of the maximal ideal in powers of nearly normally torsion-free monomial ideals. To accomplish this, we first review the following definition.

\begin{definition} (\cite[Definition 2.1]{Claudia})
A monomial ideal $I$ in a polynomial  ring $R=K[x_1, \ldots, x_n]$ over a field $K$ is called {\it nearly normally torsion-free}
  if there exist a positive integer $k$ and a monomial prime ideal
 $\mathfrak{p}$ such that $\mathrm{Ass}_R(R/I^m)=\mathrm{Min}(I)$ for all $1\leq m\leq k$, and 
 $\mathrm{Ass}_R(R/I^m) \subseteq \mathrm{Min}(I) \cup \{\mathfrak{p}\}$ for all $m \geq k+1$.
\end{definition}


We now state and prove the main result of this section in the following theorem.

\begin{theorem}\label{NNTF}  
Let $I \subset R = K[x_i : x_i\in \mathrm{supp}(I)]$ be a nearly normally torsion-free monomial ideal, $\mathfrak{m}=(x_i : x_i\in \mathrm{supp}(I))$, and also   there exist  a monomial prime ideal $\mathfrak{p}\notin \mathrm{Min}(I)$ 
and an integer $m \ge 1$ such that  the following statements hold:
\begin{enumerate}
\item[(i)]  $\operatorname{Ass}(R/I^t) = \operatorname{Min}(I)$ for all $1 \le t \le m$;
\item[(ii)]   $\operatorname{Ass}(R/I^t) \subseteq \operatorname{Min}(I) \cup \{\mathfrak{p}\}$ for all $t \ge m+1$;
\item[(iii)]  For any  monomial prime ideal $\mathfrak{P} \subsetneq \mathfrak{m}$ with $\mathfrak{P}\notin \mathrm{Min}(I)$, $I_{\mathfrak{P}}$ is normally torsion-free, where 
\(I_{\mathfrak P}\) denotes the localization of  \(I\) at  \(\mathfrak P\). 
\end{enumerate}
 If there exists some $s\geq m+1$ such that  $\mathfrak{p}\in \mathrm{Ass}(R/I^s)$, then  $\mathfrak{p}=\mathfrak{m}$.  
    \end{theorem}

\begin{proof}
Suppose that  there exists some $s\geq m+1$ such that $\mathfrak{p}\in \mathrm{Ass}(R/I^s)$. 
 For simplicity of notation, assume  $\mathrm{supp}(I)=\{x_1, \ldots, x_n\}$, $\mathfrak{m}=(x_1, \ldots, x_n)$, and 
 $\mathfrak{p}=(x_1, \ldots, x_k)$. Hence, $\mathfrak{m}\setminus \mathfrak{p}=(x_{k+1}, \ldots, x_n)$. If $I$ is  $\mathfrak{q}$-primary, then we have $\mathrm{Min}(I)=\mathrm{Ass}(R/I^t)=\{\mathfrak{q}\}$ for all $t\geq 1$. This  contradicts the  hypothesis  $\mathfrak{p}\in \mathrm{Ass}(R/I^s)\setminus \mathrm{Min}(I)$. Thus, we assume that  $I$ is not primary, and since 
$\mathrm{Min}(I)=\mathrm{Ass}(R/I)$, we must have   $|\mathrm{Min}(I)|\geq 2$. 
 Suppose that  $\mathrm{Min}(I)=\{\mathfrak{q}_1, \ldots, \mathfrak{q}_\theta\}$, where $\theta\geq 2$. 
    If $\mathfrak{p}=\mathfrak{m}$, then there is nothing to prove. Therefore, we assume that  $\mathfrak{p}\subsetneq \mathfrak{m}$. 
 Due to  $\mathrm{Min}(I)=\mathrm{Ass}(R/I)$, one can deduce from Proposition \ref{Pro.supp} that 
 $$\bigcup_{j=1}^{\theta} \mathrm{supp}(\mathfrak{q}_j)=\mathrm{supp}(I)=\mathrm{supp}(\mathfrak{m})=\{x_1, \ldots, x_n\}.$$ 
 The assumption  $\mathfrak{p}\subsetneq \mathfrak{m}$ permits us, without loss of generality, to assume   $\mathfrak{q}_j\subset \mathfrak{p}$ for all $j=1, \ldots, \lambda$, where $\lambda\geq 1$,   and $\mathfrak{q}_j\not\subset\mathfrak{p}$ for all $j=\lambda+1, \ldots, \theta$.  
 Now, consider  $$\rho:=\prod_{x_{\alpha}\in \mathfrak{m}\setminus \mathfrak{p}}x_{\alpha}=x_{k+1}  \cdots x_n.$$ 
 It follows from   $\mathfrak{p}\subsetneq \mathfrak{m}$    that $\rho \neq 1$.    Localize at $\rho$. 
    Then, one can  deduce from   \cite[Lemma 9.38]{sharp} that 
   $$\mathrm{Ass}_{R_\rho}(R_\rho/I^sR_\rho)=\{\mathfrak{q}R_\rho : \mathfrak{q} \in \mathrm{Ass}(R/I^s) \; \text{and} \; \mathfrak{q} \cap \Gamma
   =\emptyset\},$$ 
   where $\Gamma=\{\rho^n : n \geq 0\}$ is the  multiplicatively closed subset of $R$. It is straightforward to check that 
    $$\mathrm{Min}(IR_{\rho}) =\mathrm{Ass}_{R_\rho}(R_\rho/IR_\rho)=\{\mathfrak{q}_1R_\rho, \ldots, \mathfrak{q}_\lambda R_\rho\}$$ and 
   $\mathfrak{p}R_\rho\in \mathrm{Ass}_{R_\rho}(R_\rho/I^sR_\rho)$.     Note that $R_{\mathfrak{p}}$ is a further localization of $R_{\rho}$, specifically, due to  \cite[Exercise 5.45]{sharp}, we have the following  ring isomorphism $$R_{\mathfrak{p}} \cong (R_{\rho})_{\mathfrak{p}R_{\rho}}.$$ 
    As  $\mathfrak{p}R_{\rho} \in \mathrm{Ass}_{R_{\rho}}(R_{\rho}/I^s R_{\rho})$ and $\mathfrak{p}R_{\rho} \cap (R_{\rho} \setminus \mathfrak{p}R_{\rho}) = \emptyset$,  it follows from the properties of localization that $\mathfrak{p}R_{\mathfrak{p}} \in \mathrm{Ass}_{R_{\mathfrak{p}}}(R_{\mathfrak{p}}/I^s R_{\mathfrak{p}})$.
  
  On account of  $\mathfrak{p} \subsetneq \mathfrak{m}$ and  $\mathfrak{p}\notin \mathrm{Min}(I)$, condition (iii) implies that the localization $I_{\mathfrak{p}}$ is
   normally torsion-free. By the definition of normally torsion-free ideals and given that $\mathrm{Ass}(R/I) = \mathrm{Min}(I)$, we can conclude that 
\[ \mathrm{Ass}_{R_{\mathfrak{p}}}(R_{\mathfrak{p}}/I^s R_{\mathfrak{p}}) \subseteq \mathrm{Ass}_{R_{\mathfrak{p}}}(R_{\mathfrak{p}}/I R_{\mathfrak{p}}) = \mathrm{Min}(I_{\mathfrak{p}}). \]
In light of $\mathfrak{p}R_{\mathfrak{p}} \in \mathrm{Ass}_{R_{\mathfrak{p}}}(R_{\mathfrak{p}}/I^s R_{\mathfrak{p}})$, 
 we must have $\mathfrak{p}R_{\mathfrak{p}} \in \mathrm{Min}(I_{\mathfrak{p}})$, which implies that $\mathfrak{p} \in \mathrm{Min}(I)$. However, this directly contradicts the initial hypothesis that $\mathfrak{p} \notin \mathrm{Min}(I)$. Thus, our assumption that $\mathfrak{p} \subsetneq \mathfrak{m}$
  is false. Consequently, we obtain  $\mathfrak{p} = \mathfrak{m}$, and the proof is complete.
\end{proof}


It should be noted that condition (iii) in Theorem  \ref{NNTF} is essential and cannot be omitted. The following examples demonstrate  the necessity of this condition.  To this end, we first recall the following definition  from \cite{NQKR}. 

\bigskip
 Let $R=K[x_1, \ldots, x_n]$ be a polynomial ring over a field $K$. Let $t$ be a positive integer. A monomial $x_{i_1} x_{i_2} \cdots x_{i_d} \in R$ with $i_1 \leq i_2 \leq \cdots  \leq i_d$ is called {\it $t$-spread} if $i_j -i_{j-1} \geq t$ for all $j=2, \ldots, d$.  A monomial ideal in $R$ is called a {\it $t$-spread monomial ideal} if it is generated by $t$-spread monomials.   

Let $I\subset R$ be a $t$-spread monomial ideal. Then $I$ is called a {\it $t$-spread strongly stable ideal} if for all $t$-spread monomials $u\in \mathcal{G}(I)$, all $j\in \mathrm{supp}(u)$
 and all $1\leq i <j$ such that $x_i(u/x_j)$  is a $t$-spread monomial, it follows that $x_i(u/x_j)\in I$. 
 
 \begin{definition}(\cite[Definition 5.1]{NQKR})
 A monomial ideal $I\subset R$ is called a {\it $t$-spread principal Borel} if there exists a $t$-spread monomial $u\in \mathcal{G}(I)$ such that  $I$  is the smallest $t$-spread strongly stable ideal which contains $u$. We denote it as $I=B_t(u)$. It should be noted that for a $t$-spread monomial $u=x_{i_1} x_{i_2} \cdots x_{i_d} \in R$, we have $x_{j_1} x_{j_2} \cdots x_{j_d}\in \mathcal{G}(B_t(u))$ if and only if 
 $j_1\leq i_1, \ldots, j_d \leq i_d$ and $j_k - j_{k-1} \geq t$ for $k\in \{2, \ldots, d\}$.  
 \end{definition}


\begin{example}
\em{
Let $t=3$, $a=1,$ $b=7$, and $n=10$. Then  $u=x_1x_7x_{10}$ is  a $3$-spread monomial and 
$I=B_t(u)\subset R=K[x_1, x_4, x_5,x_6, x_7, x_8,x_9, x_{10}]$.  Indeed, direct computation gives that 
 \begin{align*}
 I=(&x_1x_4x_7,\; x_1x_4x_8,\; x_1x_4x_9,\; x_1x_4x_{10},\; x_1x_5x_8,\; x_1x_5x_9,\; x_1x_5x_{10},\\
 & x_1x_6x_9,\; x_1x_6x_{10}, \; x_1x_7x_{10}).
 \end{align*}
 According to Lemma \ref{Lemma 5.15}(ii), we can deduce that $I$  is nearly normally torsion-free. 
 On the other hand, by virtue of \cite[Corollary 2.6]{AEL}, we know that $I=B_t(u)$  satisfies the persistence property.  
  Using Macaulay2 \cite{GS}, we obtain 
\begin{align*}
\mathrm{Ass}(R/I) = \{& (x_1),\ (x_4,x_5,x_6,x_7),\ (x_4,x_5,x_6,x_{10}), (x_4,x_5,x_9,x_{10}),\\ 
&(x_4,x_8,x_9,x_{10}),\ (x_7,x_8,x_9,x_{10})\},
\end{align*}
and, for all $k\geq 2$, 
\begin{align*}
\mathrm{Ass}(R/I^k) = \{& (x_1),\ (x_4,x_5,x_6,x_7),\ (x_4,x_5,x_6,x_{10}), (x_4,x_5,x_9,x_{10}),\\
& (x_4,x_8,x_9,x_{10}),\ (x_7,x_8,x_9,x_{10}), (x_4,x_5,x_6,x_7,x_8,x_9,x_{10})\}.
\end{align*}
The computation above shows that, for all $k\geq 2$, we always have 
 $$\mathfrak{m}=(x_1, x_4, x_5,x_6, x_7, x_8,x_9, x_{10})\notin \mathrm{Ass}(R/I^k).$$
 In addition, due to  $(x_4, x_5,x_6, x_7, x_8,x_9, x_{10})\in \mathrm{Ass}(R/I^2)$ and  $I$  is nearly normally torsion-free, we must have 
 $\mathfrak{p}=(x_4, x_5,x_6, x_7, x_8,x_9, x_{10})$.
 
 Now, consider $\mathfrak{P}=(x_4, x_5,x_6, x_7, x_8,x_9, x_{10})\subsetneq \mathfrak{m}$. It is easy to check  
  $$I_{\mathfrak{P}}=(x_4x_7,\; x_4x_8,\; x_4x_9,\; x_4x_{10},\; x_5x_8,\; x_5x_9,\; x_5x_{10},\;
 x_6x_9,\; x_6x_{10}, \; x_7x_{10}).$$
  Using Macaulay2 \cite{GS}, we see that  $|\mathrm{Ass}(I_{\mathfrak{P}})|=5$ while $|\mathrm{Ass}(I_{\mathfrak{P}}^2)|=6$. This yields that 
 $I_{\mathfrak{P}}$ is not normally torsion-free.   
This verifies that  condition (iii)  in Theorem \ref{NNTF} is necessary  and cannot be dropped.  
}
\end{example}


 Before presenting the next example,  recall that a graph $G$ is called  \textit{almost bipartite} if  $G$ has only one induced odd cycle subgraph. 
 
 \begin{example}
\em{
Let $R=K[x_1, \ldots, x_8]$ be the polynomial ring and consider the following monomial ideal
\begin{align*}
I = (&x_2x_4x_5x_7,\,x_2x_3x_5x_7,\,x_1x_3x_5x_7,\,x_2x_4x_5x_6x_8, x_1x_3x_4x_6x_8,\\
& x_1x_2x_4x_6x_8,\,x_1x_3x_4x_6x_7,\,x_1x_2x_4x_6x_7). 
\end{align*}
 Using Macaulay2 \cite{GS}, we get the following 
\begin{align*}
\operatorname{Ass}(R/I)=\{&
(x_2,x_1),\,(x_5,x_1),\,(x_3,x_2),\,(x_4,x_3),\,(x_5,x_4),\\
&(x_7,x_4),\,(x_6,x_5),\,(x_7,x_6),\,(x_8,x_7)\},
\end{align*}
while
\begin{align*}
\operatorname{Ass}(R/I^2)=\{&(x_2,x_1),\,(x_5,x_1),\,(x_3,x_2),\,(x_4,x_3),\,(x_5,x_4),\,(x_7,x_4),\\
&(x_6,x_5),\,(x_7,x_6),\,(x_8,x_7),\,(x_5,x_4,x_3,x_2,x_1)\}.
\end{align*}
Here, consider the graph $H$ shown in Figure~2. It is straightforward to verify that the square-free monomial ideal $I$ is 
precisely the cover ideal of the graph $H$. Furthermore, we see that $H$ is an almost bipartite graph. Hence, 
one can immediately deduce from Theorem  \ref{Theorem 4.5}  that $I$ is nearly normally torsion-free.
 Since $(x_1,x_2,x_3,x_4,x_5)\in \operatorname{Ass}(R/I^2)$ and $I$ is nearly normally torsion-free, we must have 
 $\mathfrak{p}=(x_1,x_2,x_3,x_4,x_5)$. In particular, we see that $\mathfrak{p}\neq \mathfrak{m}$. 
 
 Now, take $\mathfrak{P}=(x_1,x_2,x_3,x_4,x_5)\subsetneq \mathfrak{m}$. It is not hard to investigate that 
$$I_{\mathfrak{P}}= (x_2x_4x_5,\,x_2x_3x_5,\,x_1x_3x_5,\,x_1x_3x_4,\,x_1x_2x_4).$$
  Using Macaulay2 \cite{GS}, we obtain  $|\mathrm{Ass}(I_{\mathfrak{P}})|=5$ while $|\mathrm{Ass}(I_{\mathfrak{P}}^2)|=6$. This means 
 that  $I_{\mathfrak{P}}$ is not normally torsion-free.   
Hence, condition (iii) in Theorem~\ref{NNTF} cannot be removed.

}
 \end{example}

\begin{figure}[h]
\centering
\scalebox{1.2} 
{
\begin{pspicture}(0,-1.208125)(4.2228127,1.208125)
\psdots[dotsize=0.12](1.6609375,0.6896875)
\psdots[dotsize=0.12](2.4409375,0.1096875)
\psdots[dotsize=0.12](0.8209375,0.1096875)
\psdots[dotsize=0.12](1.2409375,-0.6903125)
\psdots[dotsize=0.12](2.0609374,-0.6903125)
\psdots[dotsize=0.12](3.2209375,0.1096875)
\psdots[dotsize=0.12](2.8409376,-0.6903125)
\psline[linewidth=0.04cm](0.8009375,0.1096875)(1.6409374,0.6896875)
\psline[linewidth=0.04cm](0.8209375,0.1096875)(1.2209375,-0.6503125)
\psline[linewidth=0.04cm](1.2209375,-0.6703125)(1.2809376,-0.6703125)
\psline[linewidth=0.04cm](1.2409375,-0.6703125)(1.2409375,-0.6903125)
\psline[linewidth=0.04cm](1.2609375,-0.6703125)(1.2209375,-0.6503125)
\psline[linewidth=0.04cm](1.2409375,-0.6903125)(2.0609374,-0.6903125)
\psline[linewidth=0.04cm](2.0409374,-0.6903125)(2.4409375,0.1296875)
\psline[linewidth=0.04cm](1.6609375,0.7096875)(2.4409375,0.1096875)
\psline[linewidth=0.04cm](2.4409375,0.1096875)(3.2009375,0.1096875)
\psline[linewidth=0.04cm](2.0409374,-0.6903125)(2.0409374,-0.6703125)
\psline[linewidth=0.04cm](2.0409374,-0.6703125)(2.0409374,-0.6503125)
\psline[linewidth=0.04cm](2.0609374,-0.6903125)(2.8609376,-0.6903125)
\psline[linewidth=0.04cm](2.8409376,-0.6703125)(3.2209375,0.1296875)
\psline[linewidth=0.04cm](2.8609376,-0.6703125)(2.8609376,-0.6503125)
\psline[linewidth=0.04cm](2.8209374,-0.6903125)(3.6209376,-0.6903125)
\psdots[dotsize=0.12](3.6209376,-0.6903125)
\usefont{T1}{ptm}{m}{n}
\rput(1.6823437,1.0196875){$x_1$}
\usefont{T1}{ptm}{m}{n}
\rput(0.40234375,0.1396875){$x_2$}
\usefont{T1}{ptm}{m}{n}
\rput(1.1423438,-0.9803125){$x_3$}
\usefont{T1}{ptm}{m}{n}
\rput(2.1023438,-0.9603125){$x_4$}
\usefont{T1}{ptm}{m}{n}
\rput(2.8823438,-0.9403125){$x_7$}
\usefont{T1}{ptm}{m}{n}
\rput(2.6423438,0.4396875){$x_5$}
\usefont{T1}{ptm}{m}{n}
\rput(3.6823437,0.1396875){$x_6$}
\usefont{T1}{ptm}{m}{n}
\rput(3.7423437,-0.9403125){$x_8$}
\end{pspicture} 
}
\caption{The  graph $H$.}
\end{figure}

 
 We conclude this section with the following example, which shows    how Theorem \ref{NNTF}  can be used.
 
 \begin{example} \label{APP.2}
 \em{ 
Suppose the following monomial ideal $I$ in $R=K[x_1, \ldots, x_5]$ 
\[
I=(x_1x_4,\;x_1x_3^2,\;x_3^2x_5^2,\;x_2^3x_4,\;x_2^3x_5^2). 
\]
By computations in  Macaulay2 \cite{GS},  we obtain the following minimal primary decomposition
$$I=(x_1,x_2^3,x_3^2)\cap (x_2^3,x_3^2,x_4)\cap (x_3^2,x_4,x_5^2)  \cap (x_4,x_5^2,x_1)\cap (x_5^2,x_1,x_2^3).$$
Hence, we can deduce that 
$$\mathrm{Ass}(R/I)=\{(x_1,x_2,x_3),(x_1,x_2,x_5),(x_1,x_4,x_5),(x_2,x_3,x_4),(x_3,x_4,x_5)\}.$$
In the sequel, we  show that $I$ is nearly normally torsion-free. To do this, we rely on Lemma \ref{Lemma 3.2}. Direct computations imply the following 
statements:
\begin{itemize}
\item[$\bullet$] $I(\mathfrak{m}\setminus \{x_1\})=(x_2^3x_5^2, x_4, x_3^2).$
\item[$\bullet$] $I(\mathfrak{m}\setminus \{x_2\})=(x_1x_3^2, x_4, x_5^2).$
\item[$\bullet$] $I(\mathfrak{m}\setminus \{x_3\})=(x_2^3x_4, x_1, x_5^2).$
\item[$\bullet$] $I(\mathfrak{m}\setminus \{x_4\})=(x_3^2x_5^2, x_1, x_2^3).$
\item[$\bullet$] $I(\mathfrak{m}\setminus \{x_5\})=(x_1x_4, x_2^3, x_3^2).$
\end{itemize}
It follows immediately from Theorem \ref{NTF.Th.2.5} that $I(\mathfrak{m}\setminus \{x_i\})$, for all $1\leq i \leq 5$, is normally torsion-free. 
Hence, Lemma \ref{Lemma 3.2} yields that $I$ is nearly normally torsion-free.  This means that there exist  a monomial prime ideal $\mathfrak{p}\notin \mathrm{Min}(I)$ and an integer $m \ge 1$ such that $\operatorname{Ass}(R/I^t) = \operatorname{Min}(I)$ for all $1 \le t \le m$, and 
  $\operatorname{Ass}(R/I^t) \subseteq \operatorname{Min}(I) \cup \{\mathfrak{p}\}$ for all $t \ge m+1$.
  
One can easily check
 $\mathrm{Ass}(R/I)=\mathrm{Min}(I)$.  Let $\mathfrak{P} \subsetneq \mathfrak{m}=(x_1,x_2,x_3,x_4,x_5)$ be a  monomial prime ideal with $\mathfrak{P}\notin \mathrm{Min}(I)$.  From now on, our aim is to show that the following monomial ideals are normally torsion-free:    
\[
\begin{array}{cccc}
(1)\; I_{(x_1,x_2,x_3,x_4)} &
(2)\; I_{(x_1,x_2,x_3,x_5)} &
(3)\; I_{(x_1,x_3,x_4,x_5)} &
(4)\; I_{(x_2,x_3,x_4,x_5)} \\[2mm]

(5)\; I_{(x_1,x_2,x_4,x_5)} &
(6)\; I_{(x_1,x_2,x_4)} &
(7)\; I_{(x_1,x_3,x_4)} &
(8)\; I_{(x_1,x_3,x_5)} \\[2mm]

(9)\; I_{(x_2,x_3,x_5)} &
(10)\; I_{(x_2,x_4,x_5)} &
&
\end{array}
\]
 
\medskip
According to the discussion above, all of the monomial ideals (1)–(5) are normally torsion-free. 
Moreover, notice that 
\[I_{(x_1,x_2,x_4)}=I_{(x_1,x_3,x_4)}=I_{(x_1,x_3,x_5)}=I_{(x_2,x_3,x_5)}=I_{(x_2,x_4,x_5)}=R,\]
and so  all of the monomial ideals (6)–(10) are trivially normally torsion-free as well.  Hence,  it follows from Theorem \ref{NNTF} 
 that if there exists  $s\geq m+1$ such that $\mathfrak{p}\in \mathrm{Ass}(R/I^s)$, then we must 
 have $\mathfrak{p}=\mathfrak{m}=(x_1, x_2, x_3, x_4, x_5)$.  
 However, using  Macaulay2 \cite{GS}, we get  $\mathrm{Min}(I)=\mathrm{Ass}(R/I)=\mathrm{Ass}(R/I^2)$ and 
 $\mathrm{Ass}(R/I^3)=\mathrm{Min}(I)\cup \{(x_1, x_2, x_3, x_4, x_5)\}$, as expected. 
  }
 \end{example}
 

\bigskip
\centerline{\bf  The conflict of interest and data availability statement}

\bigskip
We hereby declare that this manuscript has no associated  data and that there is no conflict of interest regarding its content. 
 

\bigskip
\begin{center}
{\bf  ORCID}
  \vspace{3mm}
  \hspace{2cm}  \item[\textbf{ Mehrdad Nasernejad}]: ORCID: https://orcid.org/0000-0003-1073-1934
\hspace{2cm}   \item[\textbf{Jonathan Toledo}]: ORCID: https://orcid.org/0000-0003-3274-1367
 \end{center}


\bigskip
\noindent{\bf Acknowledgments.}
First, the authors  are  deeply grateful to the two anonymous referees for a  careful reading of the manuscript  and for  their   valuable suggestions which led to 
several  improvements in the quality of this paper.  In particular, the second author, Jonathan Toledo, is partially supported by SECIHTI-CBF-2025-I-1583.


\end{document}